\documentclass{amsart}
\usepackage{latexsym}
\usepackage{amssymb}
\usepackage{amsmath}
\usepackage{bbold}
\usepackage{tikz}
\def\ni{\noindent}
\def\be{\begin{equation}}
\def\ee{\end{equation}}
\def\bC{\mathbb{C}}
\def\bR{\mathbb{R}}
\def\bH{\mathbb{H}}
\def\bN{\mathbb{N}}
\def\bZ{\mathbb{Z}}
\def\gg{\mathfrak{g}}
\def\gk{\mathfrak{k}}
\def\gp{\mathfrak{p}}
\def\gt{\mathfrak{t}}

\newtheorem{df}{Definition}[section]
\newtheorem{prop}[df]{Proposition}

\newtheorem{remark}[df]{Remark}

\newtheorem{lem}[df]{Lemma}

\begin{document}

\title[A Remark on the First Eigenvalue...Symmetric Spaces]
{A Remark on the First Eigenvalue of the Laplace Operator on $1$-forms for Compact Inner Symmetric Spaces}%
\author{Jean-Louis Milhorat}
\address{Nantes Universit\'e
, CNRS, Laboratoire de Math\'ematiques Jean Leray, LMJL, UMR 6629, F-44000 Nantes, France
}
\email{jean-louis.milhorat@univ-nantes.fr}


\thanks{\textit{Acknowledgements: The author thanks Francis Burstall for having pointed out an incorrect statement in a first version of the present paper, and for the references about spherical representations.}}
\begin{abstract} We remark that on a compact inner symmetric space $G/K$, indowed with the Riemmannian metric given by the Killing form of $G$ signed-changed,  the first (non-zero) eigenvalue of the Laplace operator on $1$-forms is the Casimir eigenvalue of the highest either long or short root of $G$, according as the highest weight of the isotropy representation is long or short. Some results for the first (non-zero) eigenvalue on functions are derived.
\end{abstract}

\maketitle
\section{Introduction} It is well-known that symmetric spaces provide examples where
the spectrum of Laplace or Dirac operators
can be (theoretically) explicitly computed. However this explicit computation is far from being simple
in general and only a few examples are known.  On the other hand, several classical results in geometry 
involve the first (non-zero) eigenvalue of those spectra, so it seems interesting to get this eigenvalue without
computing all the spectrum. The present paper is a proof of the following remark:
\begin{prop} \label{result} Let $G/K$ be a compact inner symmetric space of ``{type I}'', indowed with the Riemmannian metric given by the Killing form of $G$ signed-changed. The first eigenvalue\footnote{by Bochner's vanishing theorem, there are no harmonic $1$-forms on the symmetric spaces considered here, since their Ricci curvature is positive.} of the Laplace operator acting on $1$-forms is given by the Casimir eigenvalue of the highest either long or short root of $G$ (relative to the choice of a common maximal torus $T$ in $G$ and $K$), according as the highest weight of the isotropy representation is long of short. 
\end{prop}

\noindent Note that, although the result involves the choice of a basis of roots, it does not depend on this particular choice, by the transitivity of the Weyl group $W_G$ of $G$ on root bases. Indeed, by the Freudenthal formula, the Casimir eigenvalue of the highest (long of short) root $\beta$ is given by
$$\langle \beta+2\delta_G,\beta\rangle\,,$$
where $\delta_G$ is the half-sum of the positive roots of $G$, and $\langle\,,\,\rangle$ the scalar product on the set of weights induced by the Killing form of $G$ signed-changed. Hence, by the $W_G$-invariance of the scalar product, two choices of a basis of roots (relative to the choice of a common maximal torus $T$ in $G$ and $K$) will lead to the same Casimir eigenvalue.
On the other hand, recall that for the symmetric spaces considered here, the group $G$ is simple, hence at most two lenghts occur in the sets of roots (cf. for instance \cite{Hum}). So, if only one lenght occurs, the Casimir eigenvalue of the highest root has only to be considered.

\noindent The study of subgroups of maximal rank in a compact Lie group was initiated by A. Borel and J. De Siebenthal
in \cite{BdS}, with an explicit description for compact simple groups, resulting in the following complete list of 
irreducible compact simply-connected Riemannian inner symmetric spaces $G/K$ of type I (cf. J. A. Wolf's book \cite{Wol}), where the first eigenvalue is given.

\begin{center}
 \begin{tabular}{|c|c|c|c|}
 \hline $G/K$ & \tiny{Number of} & \tiny{Lenght of the highest} & \tiny{First non-zero eigenvalue}\\
  &\tiny{$G$-root}  & \tiny{weight of the} & \tiny{of the Laplacian}\\
  &\tiny{lenghts} & \tiny{isotropy representation} &\tiny{on $1$-forms}\\
 \hline $\frac{\mathrm{SU}(p+q)}{\mathrm{S}(\mathrm{U}(p)\times \mathrm{U}(q))}$, $1\leq p\leq q$ &$1$ & long & $1$\\
 \hline $\frac{\mathrm{SO}(2p+2q+1)}{\mathrm{SO}(2p)\times \mathrm{SO}(2q+1)}$, $p\geq 1$, $q\geq 0$ & $2$ &short if $q=0$ 
 & $\frac{p}{2p-1}$\, \tiny{($G/K=S^{2p})$}\\
 &&long if $q\geq 1$ & $1$ \\
 \hline $\frac{\mathrm{Sp}(p+q)}{\mathrm{Sp}(p)\times\mathrm{Sp}(q)}$, $1\leq p\leq q$ & $2$& short &$\frac{p+q}{p+q+1}$\\
 \hline $\frac{\mathrm{Sp}(n)}{\mathrm{U}(n)}$ & $2$& long & $1$\\
 \hline $\frac{\mathrm{SO}(2p+2q)}{\mathrm{SO}(2p)\times \mathrm{SO}(2q)}$, $1\leq p \leq q$ & $1$& long & $1$\\
 \hline $\frac{\mathrm{SO}(2n)}{\mathrm{U}(n)}$, $n>2$ & $1$& long & $1$\\
 \hline $\frac{\mathrm{G}_2}{\mathrm{SO}(4)}$ & $2$& long & $1$ \\
 \hline $\frac{\mathrm{F}_4}{\mathrm{Sp}(3)\cdot \mathrm{Sp}(1)}$ & $2$&long &$1$\\
 \hline $\frac{\mathrm{F}_4}{\mathrm{Spin}(9)}$ & $2$& short& $2/3$\\
 \hline $\frac{\mathrm{E}_6}{\mathrm{SO}(10)\cdot \mathrm{SO}(2)}$ & $1$&long&$1$\\
 \hline $\frac{\mathrm{E}_6}{\mathrm{SU}(6)\cdot \mathrm{SU}(2)}$ & $1$&long&$1$\\
 \hline $\frac{\mathrm{E}_7}{\mathrm{E}_6 \cdot \mathrm{SO}(2)}$ & $1$&long&$1$\\
 \hline $\frac{\mathrm{E}_7}{\mathrm{SU}(8)/\{\pm I\}}$ & $1$&long&$1$\\
 \hline $\frac{\mathrm{E}_7}{\mathrm{SO}'(12)\cdot \mathrm{SU}(2)}$ & $1$&long&$1$\\
 \hline $\frac{\mathrm{E}_8}{\mathrm{SO}'(16)}$ & $1$&long&$1$\\
 \hline $\frac{\mathrm{E}_8}{\mathrm{E}_7\cdot \mathrm{SU}(2)}$ & $1$&long&$1$\\
 \hline
 \end{tabular}

\end{center}
(By convention, if all roots have same lenght, they are called long. The notation $\mathrm{SO}'(n)$ is used to mention that $\mathrm{SO}(n)$ acts by means of a spin representation). The (rather puzzling) fact that the eigenvalue is equal to $1$ in most of the cases is explained below.

\noindent Some results for the spectrum on functions may be derived. Indeed, if a function $f$ verifies $\Delta f=\lambda\, f$, where $\lambda$ is the first (nonzero) eigenvalue, then $\Delta df=\lambda \,df$, hence $\lambda\geq \mu$, where $\mu$ is the first eigenvalue on $1$-forms. We classify in the following the symmetric spaces for which this inequality is an equality in all the cases considered here.

\noindent It can be checked that the values given in the above table agree with already konwn results (mainly on the spectrum of functions): compare with the table given in \cite{Nag61}, with the explicit computations of the whole spectrum given in \cite{CaWol76} and \cite{Bes78}
for Compact Rank One Symmetric Spaces: $\bR P^{n}$, $\bC P^{n}$, $\bH P^{n}$, $\mathbb{C a} P^2=\mathrm{F}_4/\mathrm{Spin}(9)$, with the partial results (not always very explicit) obtained for the spectrum of Grassmannians (\cite{IT78}, \cite{Str80b}, \cite{Tsu}, \cite{TK03a}, \cite{ECh04},  \cite{H07},  \cite{ECh12}) and the spectrum of $\mathrm{Sp}(n)/\mathrm{U}(n)$,  (\cite{TK03}, \cite{HC11})\footnote{Many references for the explicit computations of spectra may be found in\\ https://mathoverflow.net/questions/219109/explicit-eigenvalues-of-the-laplacian.}.

\section{Preliminaries for the proof}
We consider a compact simply connected irreducible symmetric space $G/K$ of ``type I'',
where $G$ is a simple compact and simply-connected Lie group and $K$ is the
connected subgroup formed by the fixed elements of an involution
$\sigma$ of $G$.  This involution induces the
Cartan decomposition of the Lie algebra $\gg$ of $G$ into
$$\gg=\gk\oplus\gp\,,$$ where $\gk$ is the Lie algebra of $K$ and
$\gp$ is the vector space $\{X\in\gg\,;\, \sigma_{*}\cdot X=-X\}$.
This space $\gp$ is
canonically identified with the tangent space to $G/K$ at the
point $o$, $o$ being the class of the neutral element of $G$.
We consider here irreducible symmetric spaces, that is, the isotropy representation
$$\begin{array}{crcl}
\rho :&K&\longrightarrow& \mathrm{GL}(\gp)\\
& k&\longmapsto& \mathrm{Ad}(k)_{|\gp}
\end{array}
$$
is irreducible. Hence all $G$-invariant scalar products on $\gp$, and so all
 $G$-invariant Riemannian metrics on $G/K$ are proportional. We
consider the metric induced by the Killing form of $G$
sign-changed. With this metric, $G/K$ is an Einstein space with
scalar curvature $\mathrm{Scal}=n/2$, (cf. for instance
Theorem~7.73 in \cite{Bes}).

\ni As $G/K$ is an homogeneous space, the bundle of $p$-forms on $G/K$ may be identified with the bundle 
$G\times_{\wedge^{p}\rho} \wedge^{p} \gp$. Any $p$-form $\omega$ is then identified with a $K$-invariant function $
G\rightarrow \wedge^{p} \gp$, that is a function verifying 
$$\forall g\in G\,,\quad \forall k\in K\,,\quad \omega(gk)=\wedge^{p}\rho(k^{-1})\, \omega(g)\,.$$
Let $L_{K}^{2}(G,\wedge^{p} \gp)$ be the Hilbert space of $L^{2}$
$K$-equivariant functions $G\rightarrow \wedge^{p} \gp$. The Laplacian
operator $\Delta_p$ extends to a self-adjoint operator on
$L_{K}^{2}(G,\wedge^{p} \gp)$. Since it is an elliptic operator, it has a
(real) discrete spectrum. By the Peter-Weyl theorem,
 the natural unitary representation of $G$ on the Hilbert space
  $L_{K}^{2}(G,\wedge^{p} \gp)$
 decomposes into the Hilbert sum
 $$\bigoplus_{\gamma \in \widehat{G}}V_{\gamma}\otimes
\mathrm{Hom}_{K}(V_{\gamma},\wedge^{p} \gp)\,,$$ where $\widehat{G}$ is
the set of equivalence classes of irreducible unitary complex
representations of $G$, $(\rho_{\gamma},V_{\gamma})$ represents an
element $\gamma\in\widehat{G}$ and
$\mathrm{Hom}_{K}(V_{\gamma},\wedge^{p} \gp)$ is the vector space of
$K$-equivariant homomorphisms $V_{\gamma}\rightarrow \wedge^{p} \gp$, i.e.
$$\mathrm{Hom}_{K}(V_{\gamma},\wedge^{p} \gp)=\{ A\in
\mathrm{Hom}(V_{\gamma},\wedge^{p} \gp)\; \mathrm{s.t.}\; \forall k\in
K\,, A\circ \rho_{\gamma}(k)=\wedge^{p}\rho(k)\circ A\}\,.$$ The
injection $ V_{\gamma}\otimes
\mathrm{Hom}_{K}(V_{\gamma},\wedge^{p} \gp)\hookrightarrow
L_{K}^{2}(G,\wedge^{p} \gp)$ is given by $$v\otimes A \mapsto
\Big(g\mapsto (A\circ\rho_{\gamma}(g^{-1})\,)\cdot v\Big)\,.$$
The Laplacian $\Delta_p$ respects the above decomposition, and its restriction to the space $V_{\gamma}\otimes
\mathrm{Hom}_{K}(V_{\gamma},\wedge^{p} \gp)$ is nothing else but the Casimir operator $\mathcal{C}_{\gamma}$ of the representation
$(\rho_{\gamma},V_{\gamma})$, (see \cite{IT78}):
\begin{equation*}
\Delta (v\otimes A)= v\otimes
(A\circ \mathcal{C}_{\gamma})\,.
\end{equation*}
But since the representation is irreducible, the Casimir operator is a
scalar multiple of identity, $\mathcal{C}_{\gamma}=c_{\gamma}\,
\mathrm{id}$, where the eigenvalue $c_{\gamma}$ only depends of
$\gamma\in\widehat{G}$. Hence the spectrum of $\Delta_p$ is the set of the $c_\gamma$ for which 
$\mathrm{Hom}_{K}(V_{\gamma},\wedge^{p} \gp)$ is non trivial. 

\ni Denote by  $\wedge^{p}\rho=\oplus^{N}_{j=1} \rho^{p}_j$, the decomposition of
the representation $K\rightarrow\wedge^{p} \gp$ into irreducible
components. Note that for $p=0$ or $1$, the decomposition has only one component, since the representations are respectively the trivial one, and the isotropy representation, which are both irreducible.

\ni Now, by the Frobenius reciprocity theorem, one has 
$$\dim(\mathrm{Hom}_{K}(V_{\gamma},\wedge^{p} \gp)=\sum_{j=1}^{N} \mathrm{mult}(\rho^{p}_j,\mathrm{Res}^{G}_{K}(\rho_{\gamma}))\,,
$$
where $\mathrm{Res}^{G}_{K}(\rho_{\gamma})$ is the restriction to $K$ of the representation $\rho_{\gamma}$.
 
\ni So, finally,
\begin{equation*}
\mathrm{Spec}(\Delta_p)=\{c_{\gamma}\;;\;
\gamma\in\widehat{G}\;\mathrm{s.t.}\;\exists j\;\mathrm{s.t.}\;
\mathrm{mult}(\rho^{p}_j,\mathrm{Res}^{G}_{K}(\rho_{\gamma}))\neq
0\}\,.\end{equation*}
In particular
\begin{align}\label{spect0}
\mathrm{Spec}(\Delta_0)&=\{c_{\gamma}\;;\;
\gamma\in\widehat{G}\;\mathrm{s.t.}\;
\mathrm{mult}(\mathrm{triv.repr.},\mathrm{Res}^{G}_{K}(\rho_{\gamma}))\neq
0\}\,,\\
\intertext{and}
\mathrm{Spec}(\Delta_1)&=\{c_{\gamma}\;;\;
\gamma\in\widehat{G}\;\mathrm{s.t.}\;
\mathrm{mult}(\rho,\mathrm{Res}^{G}_{K}(\rho_{\gamma}))\neq
0\}\,.
\end{align}

\section{Proof of the result \ref{result}.}
\ni We furthermore assume that
$G$ and $K$ have same rank and consider a fixed common maximal torus $T$.

\ni Let $\Phi$ be the set of non-zero roots of the group $G$
 with respect to $T$. According to a classical terminology, a root
 $\theta$ is called compact if the corresponding root space is contained
 in $\gk_{\mathbb{C}}$ (that is, $\theta$ is a root of $K$ with respect to $T$) and
 noncompact if the root space is contained in $\gp_{\mathbb{C}}$.
Let $\Phi_{G}^{+}$  be the set of positive roots of $G$, with respect to a choice of a basis of simple roots.
The half-sum
of the positive roots of $G$ is denoted by
$\delta_{G}$. The space of weights is endowed with the
$W_{G}$-invariant scalar product $<\,,\,>$ induced by the Killing
form of $G$ sign-changed.

\ni The symmetric spaces considerer here being irreducible, the space $\gp_{\mathbb{C}}$ is irreducible.
Let $\alpha$ be the highest weight of this representation. As the group $G$ is simple, there are at most two root lenghts, and all roots of a given lenght are conjugate under the Weyl group $W_G$ of $G$ (see \cite{Hum}, \$.10.4, lemma C). So $\alpha$ is conjugate to either the maximal root of $G$ or the highest short root. Denote by $\beta$ this root, and let $w$ be any element in  $W_G$ such that $w\cdot \alpha=\beta$. We claim that 
\begin{lem} \label{lem.mult} The multiplicity  $\mathrm{mult}(\rho,\mathrm{Res}^{G}_{K}(\rho_{\beta}))$ is $\neq
0$.
\end{lem}
As the proof differs if either $\alpha$ is long or short, we first have a glance to symmetric spaces for which $\alpha$ is short.
\subsection{Symmetric spaces for which the highest weight of the isotropy representation of $K$ is a short root}\label{sh.root}
First note that if $G$ has only one root-lenght, then the highest weight $\alpha$ of the isotropy representation is necessarily a long root.
So we only have to consider symmetric spaces $G/K$ for which $G$ has two root-lenghts. Using for instance the table~2, p. 66 in \cite{Hum}, 
we have to look to the following symmetric spaces. 
\begin{enumerate} 
 \item $\mathrm{SO}(2p+2q+1)/\mathrm{SO}(2p)\times \mathrm{SO}(2q+1)$. We consider here $G=\mathrm{Spin}(2p+2q+1)$. Identifying $\bR^{2p}$ and 
 $\bR^{2q+1}$ with the subspaces spanned respectively by $e_1,\ldots, e_{2p}$ and $e_{2p+1},\ldots,e_{2p+2q+1}$, where $(e_1,\ldots,e_{2p+2q+1})$ is the canonical basis of $\bR^{2p+2q+1}$, $K$ is the subgroup of $G$ defined by
 \begin{align*}
  \mathrm{Spin}(2p)\cdot\mathrm{Spin}(2q+1)=&\big\{\psi \in \mathrm{Spin}(2p+2q+1)\,;\, \psi=\varphi\,\phi\,,\\
  &\qquad\varphi\in \mathrm{Spin}(2p)\,,\phi\in \mathrm{Spin}(2q+1)\big\}\,,
 \end{align*}
 (Note that $K=\mathrm{Spin}(2p)$, when $q=0$).
 
\noindent  We consider the common torus of $G$ and $K$ defined
by $$T=\left\{ \sum_{k=1}^{p+q}\big(\cos
(\beta_{k})+\sin(\beta_{k})\,e_{2k-1}\cdot e_{2k}\big)\,;\,
\beta_{1} ,\ldots, \beta_{p+q}\, \in
\mathbb{R}\right\}\,.$$
The Lie algebra of $T$ is $$\gt=\left\{
\sum_{k=1}^{p+q}\beta_{k}\,e_{2k-1}\cdot e_{2k}\,;\,(\ref{spect0})
\beta_{1} ,\ldots, \beta_{p+q}\, \in
\mathbb{R}\right\}\,.$$ We denote by $(x_1, \ldots ,
x_{p+q})$ the basis of  ${\gt}^{*}$ given by $$
x_{k}\cdot \sum_{j=1}^{p+q}\beta_{j}\,e_{2j-1}\cdot
e_{2j}=\beta_{k}\,.$$ We introduce the basis $(\widehat{x}_1,\ldots , \widehat{x}_{p+q})$ of $i\,\gt^{*}$ defined by
$$\widehat{x}_k:=2i\,x_{k}\,,\quad k=1,\ldots,p+q\,.$$ A
vector $\mu\in i\, \gt^{*}$ such that
$\mu=\sum_{k=1}^{p+q} \mu_{k}\, \widehat{x}_{k}$, is
denoted by $$ \mu= (\mu_{1} ,\mu_{2},\ldots
,\mu_{p+q})\, .$$ 
The restriction to $\gt$ of the
Killing form $\mathrm{B}$ of $G$ is given by $$ \mathrm{B}(e_{2k-1}\cdot
e_{2k},e_{2l-1}\cdot e_{2l})=-8(2p+2q-1)\,\delta_{kl}\,.$$ It is
easy to verify that the scalar product on $i \, \gt^{*}$ induced
by the Killing form sign changed is given by
\begin{equation}\label{scp1}\begin{split} \forall\mu=(\mu_1,\ldots
,\mu_{p+q})\in i \, \gt^{*}\,&,\,\forall
\mu'=(\mu'_1,\ldots , \mu'_{p+q})\, \in i\, \gt^*\,,\\
<\mu ,\mu'>&=\frac{1}{2(2p+2q-1)}\,\sum_{k=1}^{p+q} \mu_k\,
\mu'_k \,.\end{split}
\end{equation}
Considering the decomposition
of the complexified Lie algebra of $G$ under the action of $T$, it
is easy to verify that $T$ is a common maximal torus of $G$ and
$K$, and that the respective roots are given by (see for instance chapter~12.4 in \cite{BH3M} for details),
\begin{align*}
&\begin{cases}\pm(\widehat{x}_{i}+\widehat{x}_{j})\,,\;
\pm(\widehat{x}_{i}-\widehat{x}_{j})\,,\; 1\leq i<j\leq
p+q\,,\\
\pm \widehat{x_i}\,,\; 1\leq i\leq p+q\,,\end{cases}
&&&\text{for } G\,,\\
&\begin{cases}\pm(\widehat{x}_{i}+\widehat{x}_{j})\,,\,
\pm(\widehat{x}_{i}-\widehat{x}_{j})\,,\,1\leq i<j\leq
p\,,\,p+1\leq i<j\leq p+q\,,\\
\pm \widehat{x_i}\,,\; p+1\leq i\leq p+q\,,\end{cases}
 &&&\text{for } K\,.
\end{align*}
We consider as sets of positive roots
\begin{align*}\Phi_{G}^{+}&=\left\{\widehat{x}_{i}+\widehat{x}_{j}\,,\,
\widehat{x}_{i}-\widehat{x}_{j}\,,\;1\leq i<j\leq
p+q\,,\; \widehat{x}_i\,,\, 1\leq i\leq p+q\right\}\,,\\ \intertext{and} \Phi_{K}^{+}&=
\left\{\begin{cases}\widehat{x}_{i}+\widehat{x}_{j}\,,\\
\widehat{x}_{i}-\widehat{x}_{j}\,,\end{cases}\;\begin{cases}1\leq i<j\leq p\,,\\p+1\leq i<j\leq p+q\,,\end{cases}\;
\widehat{x}_i\,,\, p+1\leq i\leq p+q
 \right\}\,.
\end{align*}
so the set of positive non-compact roots is
$$\Phi^{+}_{\gp}=\left\{\begin{cases}\widehat{x}_{i}+\widehat{x}_{j}\,,\\
\widehat{x}_{i}-\widehat{x}_{j}\,,\end{cases}\; 1\leq i\leq p\,,\, p+1\leq j \leq p+q\,,\;\widehat{x}_i\,,\, 1\leq i\leq p\right\}\,.$$
Note that the sets
\begin{align*}
\Delta_{G}&=\{\widehat{x}_i-\widehat{x}_{i+1}\,,\;1\leq i\leq p+q-1\,, \widehat{x}_{p+q}\}\\
\intertext{and}
\Delta_{K}&=\left\{\begin{cases}\widehat{x}_i-\widehat{x}_{i+1}\,,\;1\leq i\leq p-1\,,\; \widehat{x}_{p-1}+\widehat{x}_{p}\,,\\
\widehat{x}_i-\widehat{x}_{i+1}\,,\;p+1\leq i\leq p+q-1\,,\; \widehat{x}_{p+q}\,,\end{cases}
\right\}\,,
\end{align*}
are basis of $G$-roots and $K$-roots respectively. So,

\noindent Any $\mu=(\mu_1,\ldots,\mu_{p+q})\in i\, \gt^{*}$ 
\begin{itemize}
 \item is a dominant $G$-weight if and only if
 $$\mu_1\geq \mu_2\geq \cdots \geq \mu_{p+q}\geq 0\,,$$
 and the $\mu_i$ are all simultaneously integers or half-integers,
 \item is a dominant $K$-weight if and only if
 $$\begin{cases}\mu_1\geq \mu_2\geq \cdots \geq \mu_{p-1}\geq |\mu_p|\,,\\
    \mu_{p+1}\geq \mu_{p+2}\geq \cdots \geq \mu_{p+q}\geq 0\,,
   \end{cases}$$
 and the $\mu_i$, for $1\leq i\leq p$ or $p+1\leq i\leq p+q$,  are all simultaneously integers or half-integers.
\end{itemize}
Hence 
\begin{itemize}
 \item If $q=0$, the highest weight of $\gp_{\bC}$ is the short root $\alpha=\widehat{x_1}$, (which is also the highest shortest root $\beta$ of $G$).
 \item If $q>0$, the highest weight of $\gp_{\bC}$ is the long root $\alpha=\widehat{x_1}+\widehat{x}_{p+1}$.
\end{itemize}
\item $\mathrm{Sp}(p+q)/\mathrm{Sp}(p)\times\mathrm{Sp}(q)$. The space $\mathbb{H}^{p+q}$ is viewed as a right vector
space on $\mathbb{H}$ in such a way that $G$ may be identified
with the group $$ \left\{ A\in
\mathrm{M}_{p+q}(\mathbb{H})\,;\,{}^{t}AA=I_{p+q}\right\}\,,$$
acting on the left on $\mathbb{H}^{p+q}$ in the usual way. The
group $K$ is identified with the subgroup of $G$ defined by $$
\left\{ A\in \mathrm{M}_{p+q}(\mathbb{H})\,;\,A=\begin{pmatrix}
B&0\\0&C
\end{pmatrix}\,,\,{}^{t}BB=I_{p}\,,\,\,{}^{t}CC=I_{q} \right\}\,.$$
Let $T$ be the common torus of $G$
and $K$ 
\begin{equation}\label{maxT}T:=\left\{\begin{pmatrix}
\mathrm{e}^{\mathbf{i}\beta_{1}}&&\\ &\ddots&\\ &&
\mathrm{e}^{\mathbf{i}\beta_{p+q}}\end{pmatrix}\;,\;  \beta_{1}
,\ldots, \beta_{p+q}\, \in \mathbb{R} \right\}\; ,
\end{equation}where
$$\forall \beta \in \mathbb{R}\,,\quad
\mathrm{e}^{\mathbf{i}\beta}:=\cos(\beta)+\sin(\beta)\,\mathbf{i}\,,$$
$(1,\mathbf{i},\mathbf{j},\mathbf{k})$ being the standard basis of
$\mathbb{H}$.

\noindent The Lie algebra of $T$ is $$\gt=\left\{
\begin{pmatrix} \mathbf{i}\beta_{1}&&\\ &\ddots&\\ &&
\mathbf{i}\beta_{p+q}\end{pmatrix}\; ;\; \beta_{1}, \beta_{2},
,\ldots, \beta_{p+q}\, \in \mathbb{R} \right\}\; .$$ We denote by(\ref{spect0})
$(x_1, \ldots , x_{p+q})$ the  basis of  ${\gt}^{*}$ given by $$
x_k\cdot\begin{pmatrix} \mathbf{i}\beta_{1}&&\\ &\ddots&\\ &&
\mathbf{i}\beta_{p+q}\end{pmatrix}= \beta_{k}\, .$$ A vector
$\mu\in i\, \gt^{*}$ such that $\mu=\sum_{k=1}^{p+q} \mu_{k}\,
\widehat{x}_{k}$, in the basis\\  $(\widehat {x}_k\equiv i\,
x_k)_{k=1,\ldots ,p+q}$, is denoted by $$ \mu=
(\mu_{1},\mu_{2},\ldots ,\mu_{p+q})\, .$$
The restriction to $\gt$
of the Killing form $\mathrm{B}$ of $G$ is given by $$\forall
X\in\gt\,,\,\forall Y \in\gt\,,\quad \mathrm{B}(X,Y)= 4\, (p+q+1)\,\Re\,
\big(\,\mathrm{tr} (X\,Y)\,\big)\,.$$ It is easy to verify that
the scalar product on $i \, \gt^{*}$ induced by the Killing form
sign changed is given by 
\begin{equation}\label{scp2} \begin{split}
\forall\mu=(\mu_1,\ldots ,\mu_{p+q})\in i \, \gt^{*}\,&,\,\forall
\mu'=(\mu'_1,\ldots , \mu'_{p+q})\, \in i\, \gt^*\,,\\ <\mu
,\mu'>&=\frac{1}{4(p+q+1)}\,\sum_{k=1}^{p+q} \mu_k\, \mu'_k
\,.\end{split}\end{equation}
Now, considering the decomposition of
the complexified Lie algebra of $G$ under the action of $T$, it is
easy to verify that $T$ is a common maximal torus of $G$ and $K$,
and that the respective roots are given by
\begin{align*}
&\begin{cases}\pm(\widehat{x}_{i}+\widehat{x}_{j})\,,\\
\pm(\widehat{x}_{i}-\widehat{x}_{i})\,,\\
1\leq i<j\leq p+q\,,\end{cases}\; && \pm 2\,\widehat{x}_{i}\,,\; 1\leq i\leq p+q
&&\text{for}\; G\,,\\
&\begin{cases}\pm(\widehat{x}_{i}+\widehat{x}_{j})\,,\\
\pm(\widehat{x}_{i}-\widehat{x}_{j})\,,\\
1\leq i<j\leq p\,,\\
p+1\leq i<j\leq p+q\,,
\end{cases}\;&& \pm 2\,\widehat{x}_{i}\,,\; 1\leq i\leq p+q
&&\text{for}\; K\,,
\end{align*}
We consider as sets of positive roots
\begin{align*}
 \Phi^{+}_G&=\left\{\widehat{x}_{i}+\widehat{x}_{j}\,,\widehat{x}_{i}-\widehat{x}_{i}\,,\, 1\leq i<j\leq p+q\,, 
2\,\widehat{x}_{i}\,,\; 1\leq i\leq p+q\right\}\\
\intertext{and}
\Phi^{+}_K&=\left\{\begin{cases}\widehat{x}_{i}+\widehat{x}_{j}\\ \widehat{x}_{i}-\widehat{x}_{j}\,,\end{cases}\,
\begin{cases}
1\leq i<j\leq p\\
p+1\leq i<j\leq p+q\,,\end{cases}\,2\,\widehat{x}_{i}\,,\; 1\leq i\leq p+q
\right\}\,,
\end{align*}
so the set of positive non-compact roots is
$$\Phi^{+}_{\gp}=\{\widehat{x}_{i}+\widehat{x}_{j}\,,\widehat{x}_{i}-\widehat{x}_{i}\,, 1\leq i\leq p\,, p+1\leq j\leq p+q
\}\,.$$
Note that the sets
\begin{align*}
 \Delta_G&=\left\{ \widehat{x}_i-\widehat{x}_{i+1}\,, 1\leq i\leq p+q-1\,,\; 2\,\widehat{x}_{p+q}\right\}\,,\\
 \intertext{and}
 \Delta_K&=\left\{ \begin{cases}\widehat{x}_i-\widehat{x}_{i+1}\,, 1\leq i\leq p-1\,, 2\,\widehat{x}_{p}\,,\\
   \widehat{x}_i-\widehat{x}_{i+1}\,,
 p+1\leq i\leq p+q-1\,,2\,\widehat{x}_{p+q}\end{cases}
\right\}\,,
\end{align*}
are basis of $G$-roots and $K$-roots respectively. So,

\noindent Any $\mu=(\mu_1,\ldots,\mu_{p+q})\in i\, \gt^{*}$ 
\begin{itemize}
 \item is a dominant $G$-weight if and only if
 $$\mu_1\geq \mu_2\geq \cdots \geq \mu_{p+q}\geq 0\,,$$
 and the $\mu_i$ are all integers,
 \item is a dominant $K$-weight if and only if
 $$\mu_1\geq \mu_2\geq \cdots \geq \mu_{p}\geq 0 \;,\;\mu_{p+1}\geq \mu_{p+2}\geq \cdots \geq \mu_{p+q}\geq 0$$
 and the $\mu_i$ are all integers.
\end{itemize}
Hence the highest weight of $\gp_{\bC}$ is the short root $\alpha=\widehat{x}_1+\widehat{x}_{p+1}$.
\item $\mathrm{Sp}(n)/\mathrm{U}(n)$. With the same notations as above, we consider the subgroup $K$ of $G=\mathrm{Sp}(n)$:
$$K=\left\{A=(a_{ij})\in \mathrm{Sp}(n)\,;\, a_{ij}\in \bR+\mathbf{i}\, \bR\right\}
  \simeq \mathrm{U}(n)\,.$$
Note that $K$ is the set of fixed points of the inner involution:
 $$\sigma:\mathrm{Sp}(n)\rightarrow \mathrm{Sp}(n)\,,\; A\mapsto \mathbf{I} A \mathbf{I}^{-1}\,,$$
 where 
 $$ \mathbf{I}=\mathbf{i}\, I_n\,.$$
 The subspace $\gp=\{A\in\mathfrak{sp}(n)\,;\, \sigma_{*}(A)=-A\}$ is then the set
 $$\gp=\left\{A=(a_{ij})\in \mathfrak{sp}(n)\,;\, a_{ij}=\mathbf{j}\, \bR+\mathbf{k}\, \bR\right\}\,.$$
 The torus $T$ introduced above (\ref{maxT}) is a common torus of $G$ and $K$. Considering the decomposition of
the complexified Lie algebra of $G$ under the action of $T$, it is
easy to verify that $T$ is a common maximal torus of $G$ and $K$,
and that the respective roots are given by
\begin{align*}
&\left\{\pm(\widehat{x}_i-\widehat{x}_j)\,,\; 1\leq i<j\leq n\,,\;\pm(\widehat{x}_i+\widehat{x}_j)\,,\; 1\leq i\leq j\leq n\right\} 
&&\text{for}\; G\,,\\
&\left\{\pm(\widehat{x}_i-\widehat{x}_j)\,,\; 1\leq i< j\leq n\right\} 
&&\text{for}\; K\,,
\end{align*}
We consider as sets of positive roots
\begin{align*}
 \Phi^{+}_G&=\left\{\widehat{x}_i-\widehat{x}_j\,,\; 1\leq i<j\leq n\,,\;\widehat{x}_i+\widehat{x}_j\,,\; 1\leq i\leq j\leq n\right\}\\
\intertext{and}
\Phi^{+}_K&=\left\{\widehat{x}_i-\widehat{x}_j)\,,\; 1\leq i< j\leq n\right\}\,,
\end{align*}
so the set of positive non-compact roots is
$$\Phi^{+}_{\gp}=\{\widehat{x}_{i}+\widehat{x}_{j}\,,\, 1\leq i\leq j\leq n\}\,.$$
Note that the sets
\begin{align*}
 \Delta_G&=\left\{ \widehat{x}_i-\widehat{x}_{i+1}\,, 1\leq i\leq n-1\,,\; 2\,\widehat{x}_{n}\right\}\,,\\
 \intertext{and}
 \Delta_K&=\left\{\widehat{x}_i-\widehat{x}_{i+1}\,, 1\leq i\leq n-1\right\}\,,
\end{align*}
are basis of $G$-roots and $K$-roots respectively (there are only $n-1$ simple $K$-roots as $K$ is not semi-simple). So, 

\noindent Any $\mu=(\mu_1,\ldots,\mu_{n})\in i\, \gt^{*}$ is(\ref{spect0})
 \begin{itemize}
 \item  a dominant $G$-weight if and only if
 $$\mu_1\geq \mu_2\geq \cdots \geq \mu_{n}\geq 0\,,$$
 and the $\mu_i$ are all integers,
 \item  a dominant $K$-weight if and only if 
 $$ \mu_{i}-\mu_{i+1}\in \bN\,,\; 1\leq i\leq n-1\,.$$
 \end{itemize}
Hence there are two dominant $K$-weights in the representation $\gp_{\bC}$: $2\,\widehat{x}_1$ and $\widehat{x}_1+\widehat{x}_2$, but the highest weight is $2\,\widehat{x}_1$ since $\widehat{x}_1+\widehat{x}_2\prec 2\,\widehat{x}_1$. 
Hence the highest weight of $\gp_{\bC}$ is the long root $\alpha=2\,\widehat{x}_1$.
\item $\mathrm{G}_2/\mathrm{SO}(4)$. We use here the results given in \cite{CG} (see page~226), which follow from a general result of Borel-de Siebenthal, \cite{BdS}. A set of $G$-roots $\Phi_G$ is given by the elements $x\in \bR^3$, whose coordinates are integers verifying
$$\sum_{i=1}^3 x_i=0 \quad \text{and} \quad \|x\|^2 = 2\; \text{or}\; 6\,,$$
hence $$\Phi_G=\{\pm(e_1-e_2)\,,\pm(e_2-e_3)\,,\pm(e_3-e_1)\,,\pm(2\,e_1-e_2-e_3)\,,\pm(2\,e_2-e_1-e_3),\pm(2\,e_3-e_1-e_2)\}\,.$$
The following system of positive $G$-roots is choosen:
$$\Phi_G^{+}=\{e_1-e_2\,, e_3-e_2\,,e_3-e_1\,,-2\, e_1+e_2+e_3\,,-2\, e_2+e_3+e_1\,,2\,e_3-e_1-e_2\}\,.$$
It can be checked that a basis of $G$-roots is given by
$$\Delta_G=\{e_1-e_2, -2\, e_1+e_2+e_3\}\,.$$
A system of positive $K$-roots (which appears to be also a basis of $K$-roots) is then given by
$$\Phi_K^{+}=\{e_1-e_2\,, -e_1-e_2+2\, e_3\}\,.$$
The set of positive non-compact roots is
$$\Phi_{\gp}^{+}=\{e_3-e_2\,,e_3-e_1\,,-2\,e_1+e_2+e_3\,,-2\, e_2+e_3+e_1\}\,.$$
There are two dominant weights in $\gp_{\bC}$: $-2\, e_2+e_3+e_1$ and $e_3-e_2$, but the highest weight is $-2\, e_2+e_3+e_1$ since
$e_3-e_2\prec -2\, e_2+e_3+e_1$. Hence the highest weight of $\gp_{\bC}$ is the long root $\alpha=-2\, e_2+e_3+e_1$.
\item $\mathrm{F}_4/\mathrm{Sp}(3)\cdot \mathrm{Sp}(1)$. We use here also results given in \cite{CG} (see page~227).
A set $\Phi_G$ of $G$-roots is given by the elements $x\in \bR^4$ whose coordinates are integers or half-integers satisfying $\|x\|^2=1$ or $2$, \cite{Hum}, hence
$$\Phi_G=\left\{\pm e_i\,,1\leq i\leq 4\,,\, \pm e_i\pm e_j\,,\,1\leq i<j\leq 4\,,\, \frac{1}{2}(\pm e_1\pm e_2 \pm e_3 \pm e_4)\right\}\,.
$$ 
We consider the system of positive roots 
$$\Phi^{+}_G=\left\{e_i\,,\, 1\leq i\leq 4\,,\, e_i\pm e_j\,,\,1\leq i<j\leq 4\,,\,
\frac{1}{2}(e_1\pm e_2 \pm e_3 \pm e_4)\right\}\,.$$
It can be check that a basis  of $G$-roots is given by
$$\Delta_G=\{\alpha_1:=e_2-e_3\,,\, \alpha_2:=e_3-e_4\,,\, \alpha_3:=e_4\,,\, \alpha_4:=\frac{1}{2}(e_1-e_2-e_3-e_4)\}\,.$$
A system of positive $K$-roots is then given by
\begin{align*}
\Phi_K^{+}=\Big\{ &e_3\,,e_4\,,e_1+e_2\,,e_1-e_2\,,e_3+e_4\,,e_3-e_4\,, \frac{1}{2}\,(e_1-e_2-e_3-e_4)\,,\\
&\frac{1}{2}\,(e_1-e_2-e_3+e_4)\,,\frac{1}{2}\,(e_1-e_2+e_3-e_4)\,,\frac{1}{2}\,(e_1-e_2+e_3+e_4)\Big\}\,.
\end{align*}
It can be check that a basis  of  $K$-roots is given by
$$\Delta_K=\left\{e_1+e_2\,,e_4\,,e_3-e_4\,, \frac{1}{2}\,(e_1-e_2-e_3-e_4)\right\}\,.$$
The set of positive non-compact roots is
\begin{align*}
 \Phi_{\gp}^{+}=\Big\{ &e_1\,,e_2\,,e_1+e_3\,,e_1+e_4\,,e_2+e_3,e_2+e_4\,, e_1-e_3\,,e_1-e_4\,,\\
 & e_2-e_3\,,e_2-e_4\,,\frac{1}{2}\,(e_1+e_2+e_3+e_4)\,,\frac{1}{2}\,(e_1+e_2+e_3-e_4)\,,\\
 & \frac{1}{2}\,(e_1+e_2-e_3+e_4)\,,\frac{1}{2}\,(e_1+e_2-e_3-e_4)\Big\}\,.
\end{align*}
There are two dominant weights in $\gp_{\bC}$: $e_1$ and $e_1+e_3$, but the highest weight is $e_1+e_3$ since 
$e_1\prec e_1+e_3$.  Hence the highest weight of $\gp_{\bC}$ is the long root $\alpha=e_1+e_3$.

\item $\mathrm{F}_4/\mathrm{Spin}(9)$. In that case, a system of positive $K$-roots is
$$\Phi^{+}_K=\{e_i\,,\, 1\leq i\leq 4\,,\, e_i\pm e_j\,,\,1\leq i<j\leq 4\}\,.$$
A basis of simple $K$-roots is given by
$$\Delta_K=\left\{e_1-e_2, e_2-e_3, e_3-e_4, e_4\right\}\,.$$(\ref{spect0})
The set of positive non-compact roots is
$$\Phi^{+}_{\gp}=\left\{\frac{1}{2}(e_1\pm e_2 \pm e_3 \pm e_4)\right\}\,,$$
hence the highest weight of the representation $\gp_{\bC}$ is the only dominant weight: the short root $\alpha= \frac{1}{2}(e_1+ e_2 + e_3 +e_4)$.

\end{enumerate}

\subsection{Proof of the lemma \ref{lem.mult}}

\begin{proof}
The proof is very simple if there is only one root lenght, or if $\alpha$ is a long root. In that case, $\beta$ is necessarily the maximal root, hence the highest weight of the adjoint representation of the simple group $G$ in its complexified Lie algebra $\gg_{\bC}$. But the decomposition of $\gg_{\bC}=\gk_{\bC}\oplus \gp_{\bC}$ into $K$-invariant subspaces implies at once that $\rho$ is contained in the restriction of $\rho_{\beta}$ to $K$.

\noindent The proof is a little more involved when two roots lenghts occur and $\alpha$ is a short root. As we saw it just before, there are only three cases to be considered here: $\mathrm{SO}(2p+1)/\mathrm{SO}(2p)$, $\mathrm{Sp}(p+q)/\mathrm{Sp}(p)\times \mathrm{Sp}(q)$, and $\mathrm{F}_4/\mathrm{Spin}(9)$.

\noindent Let $v_{\beta}$ be the maximal vector (unique up to a scalar multiple) of the representation $V_{\beta}$ and let $g\in G$ be some representant of $w^{-1}\in W_G$. First $\rho_{\beta}(g)\cdot v_{\beta}$ is a weight-vector for the weight $\alpha$ since forall $ X \in \gt$,
\begin{align*}
 \rho_{\beta *}(X)\cdot(\rho_{\beta}(g)\cdot v_{\beta})&= {\frac{d}{dt}}_{|t=0}\left(\rho_{\beta}(\exp(t\,X)g)\cdot v_{\beta}\right)\\
 &={\frac{d}{dt}}_{|t=0}\left(\rho_{\beta}(gg^{-1}\exp(t\,X)g)\cdot v_{\beta}\right)\\
 &=\rho_{\beta}(g)\cdot \rho_{\beta *}(\mathrm{Ad}(g^{-1})\cdot X)\cdot v_{\beta}\\
 &= \rho_{\beta}(g)\cdot \beta( w\cdot X)\,v_{\beta}\\
 &= w^{-1}(\beta)(X)\, \rho_{\beta}(g)\cdot v_{\beta}\\
 &= \alpha(X) \, \rho_{\beta}(g)\cdot v_{\beta}\,.
\end{align*}
Now, we claim that there exists $w\in W_G$ such that $\beta=w\cdot \alpha$, and the weight-vector $\rho_{\beta}(g)\cdot v_{\beta}$, where $g$ is a representative of $w^{-1}$, is a maximal vector for the action of $K$. We may consider only the action of a basis of simple $K$-roots $\{\theta_1,\ldots,\theta_r\}$. First note that if $\theta$ is a positive $K$-root, and $E_{\theta}$ a root-vector associated to it, then $\mathrm{Ad}(g^{-1})\cdot E_{\theta}$ is a root-vector for the root $w\cdot \theta$, since
 forall $X\in \gt$,
\begin{align*}
  [X,\mathrm{Ad}(g^{-1})\cdot E_{\theta}]&=\mathrm{Ad}(g^{-1})\cdot[\mathrm{Ad}(g)\cdot X,E_{\theta}]\\
  &= \mathrm{Ad}(g^{-1})\cdot[w^{-1}\cdot X, E_{\theta}]\\
  &= \theta(w^{-1}\cdot X)\, \mathrm{Ad}(g^{-1})\cdot E_{\theta}\\
  &= (w\cdot\theta)(X)\, \mathrm{Ad}(g^{-1})\cdot E_{\theta}\,.
 \end{align*}
Now, if $w\cdot \theta$ is a positive root, then,  since $v_{\beta}$ is a maximal vector killed by any root-vector associated to a positive root,
\begin{align*}
\rho_{\beta *}(E_{\theta})\cdot (\rho_{\beta}(g)\cdot v_{\beta})&={\frac{d}{dt}}_{|t=0}\left(\rho_{\beta}(gg^{-1}\exp(t\,E_{\theta})g)\cdot v_{\beta}\right)\\
 &=\rho_{\beta}(g)\cdot\rho_{\beta *}(\mathrm{Ad}(g^{-1})\cdot E_{\theta})\cdot v_{\beta}\\
 &=\rho_{\beta}(g)\cdot 0=0\,.
\end{align*}
So $\rho_{\beta}(g)\cdot v_{\beta}$ is killed by the action of any positive $K$-root $\theta$ such that $w\cdot \theta$ is a positive root.
Hence we may conclude by proving that, for each symmetric space under consideration, there exists $w\in W_G$ such that $w\cdot\alpha=\beta$, and 
$w\cdot \theta_i$ is a positive $G$-root, for any simple $K$-root $\theta_i$.
\begin{enumerate}
 \item $\mathrm{SO}(2p+1)/\mathrm{SO}(2p)$. We saw above that the highest weight of the representation $\gp_{\bC}$ is the short root $\alpha=\widehat{x_1}$, which is also the highest shortest root $\beta$ of $G$. We may choose $w=id$, and as all the $K$-simple roots in $\Delta_K$ are $G$-positive roots, the claim is verified in that case. 
\item $\mathrm{Sp}(p+q)/\mathrm{Sp}(p)\times \mathrm{Sp}(q)$. We saw above that the highest weight of the representation $\gp_{\bC}$ is the short root $\alpha=\widehat{x_1}+\widehat{x}_{p+1}$. Now the highest short $G$-root is
$\beta=\widehat{x}_1+\widehat{x}_{2}$. Let $w$ be the element of the Weyl group $W_G$ given by the $p$-cycle permutation 
$(2\,3\cdots p+1)$. One has $w\cdot \alpha=\beta$, and it is easily verified that $w\cdot \theta_i$ is a positive $G$-root for any simple $K$-root $\theta_i$, hence the claim is also proved in that case.
\item $\mathrm{F}_4/\mathrm{Spin}(9)$. We saw above that the highest weight of the representation $\gp_{\bC}$ is the short root $\alpha= \frac{1}{2}(e_1+ e_2 + e_3 +e_4)$. Now, the highest short $G$-root is $\beta= e_1$, and the reflection $\sigma_{\alpha_4}$ across the hyperplane $\alpha_4^{\bot}$ verifies $\sigma_{\alpha_4}\cdot \alpha=\beta$. It is easily verified that $\sigma_{\alpha_4}\cdot \theta_i$ is a positive $G$-root for any simple $K$-root $\theta_i$ since 
\begin{align*}
 \sigma_{\alpha_4}:  e_1-e_2&\mapsto e_3+e_4\,,\\
  e_2-e_3 &\mapsto e_2-e_3\,,\\
  e_3-e_4 &\mapsto e_3 -e_4\,,\\
 e_4 &\mapsto \frac{1}{2}\,(e_1-e_2-e_3+e_4)\,.
\end{align*}
Hence the claim is also proved in that case.
\end{enumerate}
\end{proof}

\subsection{First eigenvalue of the Laplace operator acting on $1$-forms}
\noindent In order to conclude, we have to verify that the Casimir eigenvalue $c_{\beta}$ is the lowest eigenvalue of the Laplacian.
\begin{lem} Let $(\rho_{\gamma}, V_{\gamma})$ be an irreducible $G$ representation such that the multiplicity  $\mathrm{mult}(\rho,\mathrm{Res}^{G}_{K}(\rho_{\gamma}))\not =0$. Then $c_{\beta}\leq c_{\gamma}$.
\end{lem}
\begin{proof}
 Recall that the highest weight of $\rho$ is $\alpha$, hence if $\mathrm{mult}(\rho,\mathrm{Res}^{G}_{K}(\rho_{\gamma}))\not =0$, then $\alpha$ and $\beta=w\cdot \alpha$ are actually weights of the representation $\rho_\gamma$. But then, as $\gamma$ is the highest weight,
 $$\langle \beta+\delta_G, \beta+\delta_G\rangle \leq \langle \gamma+\delta_G,
 \gamma+\delta_G\rangle\,,$$
 cf. Lemma~C, p.71 in \cite{Hum}, hence 
 $$c_{\beta}=\|\beta+\delta_G\|^2-\|\delta_G\|^2\leq \|\gamma+\delta_G\|^2-\|\delta_G\|^2=c_{\gamma}\,.$$
\end{proof}

 \begin{remark} The fact that $c_{\beta}=1$ when $\beta$ is the highest long root may seem to be rather puzzling. This is indeed a consequence of  Freudenthal's formula (cf. 48.2 in \cite{FdV} or p. 123 in \cite{Hum}): for any $G$-weight $\mu$ of the representation $(\rho_{\beta},V_{\beta})=(\mathrm{Ad}, \gg_{\bC})$ , the multiplicity $\mathrm{mult}(\mu)$ of $\mu$ is given recursively by the formula
 $$(\langle \beta+\delta_G,\beta+\delta_G\rangle -\langle \mu+\delta_G,\mu+\delta_G\rangle)\, \mathrm{mult}(\mu)=2\, \sum_{\theta\succ 0}\,\sum_{i=1}^{\infty}
  \mathrm{mult}(\mu+i\theta)\, \langle \mu +i\theta,\theta\rangle\,.$$
  Applying this formula to the weight $\mu=0$, one obtains since 
  \begin{itemize}
   \item $\mathrm{mult}(0)=\dim(\gt)$, as $T$ is a maximal common torus,
   \item for any integer $i\geq 1$, $\mathrm{mult}(i\theta)\not =0\Leftrightarrow i=1$, and 
   $\mathrm{mult}(\theta)=1$, by properties of roots, since the only multiple of a root $\theta$ which is itself a root is $\pm \theta$,
  \end{itemize}
 $$(\|\beta+\delta_G\|^2-\|\delta_G\|^2)\, \dim(\gt)=2\, \sum_{\theta\succ 0} \|\theta\|^2\,.$$
But Gordon Brown's formula (\cite{Bro64} or 21.5 in \cite{FdV}) states that
$$2\, \sum_{\theta\succ 0} \|\theta\|^2=\dim(\gt)\,,$$
hence $$c_{\beta}=1\,.$$
\end{remark}

\section{The first non-zero eigenvalue of the Laplace operator on functions} 

\noindent As it was noticed in the introduction, any (non-zero) eigenvalue $\lambda$ of the Laplace operator on functions is greater or equal to the first eigenvalue of the Laplace operator on $1$-forms. We now examine in which case this inequality is an equality.

\begin{prop} Let $\lambda$ be the first non-zero eigenvalue of the Laplace operator on functions, and let $\mu$ be the first eigenvalue of the Laplace operator on $1$-forms.

\noindent If $G/K$ is a Hermitian symmetric space:
\begin{equation*}\begin{split}
\mathrm{SU}(p+q)/\mathrm{S}(\mathrm{U}(p)\times \mathrm{U}(q))\quad,\quad \mathrm{SO}(p+2)/\mathrm{SO}(n)\times \mathrm{SO}(2)\quad,\quad
\mathrm{Sp}(n)/\mathrm{U}(n)\,,\\
\mathrm{SO}(2n)/\mathrm{U}(n)\quad ,\quad \mathrm{E}_6/\mathrm{SO}(10)\cdot \mathrm{SO}(2)\quad ,\quad \mathrm{E}_7/\mathrm{E}_6\cdot \mathrm{SO}(2)\,,
\end{split}
\end{equation*}
or if $G/K$ is one of the symmetric spaces for which the highest weight of the isotropy representation is a short root: 
$$\mathrm{SO}(2p+1)/\mathrm{SO}(2p)\quad ,\quad  \mathrm{Sp}(p+q)/\mathrm{Sp}(p)\times\mathrm{Sp}(q)\quad \text{or}\quad
\mathrm{F}_4/\mathrm{Sp}(3)\cdot\mathrm{Sp}(1)\,,$$ then $\lambda=\mu$.

\noindent In all the other cases, $\lambda>\mu$.
 
\end{prop}

\begin{proof}
\noindent  By (\ref{spect0}), the spectrum of the Laplace operator on functions is given by considering the irreducible representations $\gamma\in \widehat{G}$ such that $\mathrm{mult}(\mathrm{triv.repr.},\mathrm{Res}^{G}_{K}(\rho_{\gamma}))\neq 0$, or equivalently, such that there exists a non-zero weight-vector trivially acted by the group $K$. Such representations  are called spherical representations \cite{Sug62}, \cite{Tak94}.

\noindent Hence, to prove the equality $\lambda=\mu$, we have to verify that\\ $\mathrm{mult}(\mathrm{triv.repr.},\mathrm{Res}^{G}_{K}(\rho_{\beta}))\neq 0$, where $\beta$ is the highest long or short root. Indeed, if that is verified, then $c_\beta\geq \lambda$, as $c_{\beta}$ belongs to the spectrum. But, as it was remarked in the introduction, $\lambda\geq c_{\beta}$, since $\lambda$ has to be  greater or equal to the first eigenvalue of Laplace operator on $1$-forms.

\noindent Let us first examine the case when $\beta$ is the highest long root. In that case, $\rho_{\beta}$ is the adjoint representation of the simple group $G$ on its complexified algebra $\gg_{\bC}$, hence a non-zero root-vector is trivially acted by the subgroup $K$ if and only if it belongs to the center of $K$. But this is only possible if and only if $G/K$ is hermitian, see for instance \cite{Wol}.

\noindent Let us consider now the three symmetric spaces where two root lengths occur and $\beta$ is the highest short root:
$\mathrm{SO}(2p+1)/\mathrm{SO}(2p)$, $\mathrm{Sp}(p+q)/\mathrm{Sp}(p)\times \mathrm{Sp}(q)$, and 
and $\mathrm{F}_4/\mathrm{Spin}(9)$. Going back to the case-by-case study of those three symmetric spaces above, one gets:
\begin{enumerate}
 \item $\mathrm{SO}(2p+1)/\mathrm{SO}(2p)$. The highest short root is $\beta=\widehat{x}_1$, which is the highest weight of the fundamental standard representation of the group  $\mathrm{Spin}(2p+1)$ (or $\mathrm{SO}(2p+1)$) in the space $\bC^{2p+1}$, see for instance chapter~12 in 
 \cite{BH3M}. Now, the group $K$ acts trivially on the last vector $e_{2p+1}$ of the canonical basis $(e_1,\ldots,e_{2p+1})$ of $\bC^{2p+1}$
 since the inclusion $K\subset G$ is induced by the natural inclusion
 $$A\in \mathrm{SO}(2p)\longmapsto \begin{pmatrix} A & 0\\ 0 & 1\end{pmatrix}\in \mathrm{SO}(2p+1)\,.$$
 Hence $\mathrm{mult}(\mathrm{triv.repr.},\mathrm{Res}^{G}_{K}(\rho_{\beta}))\neq
0$.
\item $\mathrm{Sp}(p+q)/\mathrm{Sp}(p)\times\mathrm{Sp}(q)$. The highest short root is $\beta=\widehat{x}_1+\widehat{x}_{2}$, which is the highest weight of the fundamental representation of the group $\mathrm{Sp}(p+q)$ in the space $\wedge^{2}_{0}(\bC^{2(p+q)})$.  To be more explicit, 
 $\bH^{p+q}$ is identified with $\bC^{2(p+q)}$ here, such that $\mathrm{Sp}(p+q)\simeq \mathrm{SU}(2(p+q))\cap \mathrm{Sp}(2(p+q),\bC)$, and the representation of $\mathrm{Sp}(p+q)$ in the space $\wedge^2(\bC^{2(p+q)})$ decomposes into two irreducible pieces:
 \begin{equation}\label{decomp}
  {\wedge}^2(\bC^{2(p+q)})={\wedge}^2_{0}(\bC^{2(p+q)})\oplus \bC\cdot \omega_{p+q}\,,
 \end{equation}
 where the first piece is  the fundamental representation with highest weight $\widehat{x}_1+\widehat{x}_{2}$, and the second one the trivial representation on the space generated by the symplectic $2$-form:
 $$\omega_{p+q}:=e_1\wedge e_{-1}+e_2\wedge e_{-2}+\cdots +e_{p+q}\wedge e_{-(p+q)}\,,$$
where $(e_1,e_2,\ldots,e_{p+q},e_{-1},e_{-2},\ldots,e_{-(p+q)})$ is the canonical basis of\\ $\bC^{2(p+q)}$, see for instance chapter~12 in 
 \cite{BH3M}.

\noindent Note then that the $2$-form
$$ e_1\wedge e_{-1}+e_2\wedge e_{-2}+\cdots+ e_{p}\wedge e_{-p}\,,$$
is $K=\mathrm{Sp}(p)\times\mathrm{Sp}(q)$-invariant, so its (non-zero) component in $\wedge^2_{0}(\bC^{2(p+q)})$ under the decomposition (\ref{decomp}) is also an invariant of 
the group $K$, hence $\mathrm{mult}(\mathrm{triv.repr.},\mathrm{Res}^{G}_{K}(\rho_{\beta}))\neq
0$.

\item $\mathrm{F}_4/\mathrm{Spin}(9)$. The highest short root is $\beta= e_1$. The function ``{branch}'' in the LiE Program, \cite{Lie},
$$branch([0,0,0,1],B4,res\_mat(B4,(F4)),(F4))\,,$$
returns that the dominant weights in the decomposition of $\mathrm{Res}^{G}_{K}(\rho_{\beta})$ (expressed in the basis of fundamental weights) are
$$[0,0,0,0]\,,\quad [0,0,0,1] \quad \text{and}\quad [1,0,0,0]\,,$$
hence $\mathrm{mult}(\mathrm{triv.repr.},\mathrm{Res}^{G}_{K}(\rho_{\beta}))=1\neq
0$. 

\noindent The result may be also verified with the help of the Satake diagram of $\mathrm{F}_4/\mathrm{Spin}(9)$ (see \cite{Sug62}).
\end{enumerate}
\noindent Consider now all the remaining symmetric spaces. The highest weight of their isotropy representation is long. Let $\gamma\in \widehat{G}$ be such that $c_{\gamma}=\lambda$ and\\ $\mathrm{mult}(\mathrm{triv.repr.},\mathrm{Res}^{G}_{K}(\rho_{\gamma})\neq 0$. The irreducible representation $\gamma$ is not the adjoint representation of group $G$ on its complexified algebra $\gg_{\bC}$, since otherwise $G/K$ should be hermitian. Let $v$ be a weight-vector trivially acted by the group $K$. This vector $v$ is not killed by the action of at least one root-vector associated to a simple (non-compact) root, since otherwise $\gamma$ should not be irreducible. The corresponding positive non-compact root belongs then to the set of weights. We claim that such a root is long. This is obviously true if there is only one root-lenght. But this is also true if there are two root-lenghts: going back to the symmetric spaces considered in section~\ref{sh.root}, it can be checked that in all the cases under consideration, any short simple root of $G$ is always also a simple root of $K$. Now, since all roots of a given lenght are conjugate under the Weyl group (see for instance Lemma~C, p. 71 in \cite{Hum}), one can deduce that the highest long root $\beta$ belongs to the set of weights of $\gamma$. 
As $\beta$ is not the highest weight of $\gamma$, since otherwise $G/K$ should be hermitian, one can conclude using Lemma~C, p.71 in \cite{Hum}
that $c_{\gamma}>c_{\beta}=1$, hence $\lambda>\mu$.
\end{proof}

\noindent Francis Burstall has pointed out to us that the three cases : $\lambda<1$, $\lambda=1$ and $\lambda>1$ exactly amount the three possibilities for the second homotopy group of $G/K$ : $\lambda<1$ if $\pi_2(G/K)$ is trivial, $\lambda=1$ if $\pi_2(G/K)=\bZ$, and $\lambda>1$ if $\pi_2(G/K)=\bZ/\bZ_{2}$, see chapter~3 in \cite{BR90}.

\noindent To conclude, we have to compute the Casimir eigenvalue $c_{\beta}$ for the three symmetric spaces for which the highest weight of the isotropy representation is a short root.
\begin{enumerate}
 \item $\mathrm{SO}(2p+1)/\mathrm{SO}(2p)$. Here $\beta=\widehat{x}_1$. The half-sum of the positive $G$-roots is 
 $$\delta_G=\left(p-\frac{1}{2},p-\frac{3}{2},\ldots,\frac{3}{2},\frac{1}{2}\right)\,.$$
 Using (\ref{scp1}), one gets
 $$c_{\beta}=\langle\beta+2\,\delta_G,\beta\rangle=\frac{2\,p}{2\,(2p-1)}=\frac{p}{2p-1}\,.$$
 \item $\mathrm{Sp}(p+q)/\mathrm{Sp}(p)\times\mathrm{Sp}(q)$. Here $\beta=\widehat{x}_1+\widehat{x}_2$. The half-sum of the positive roots is
 $$\delta_G=\left(p+q,p+q-1,\ldots, 2,1\right)\,.$$
 Using (\ref{scp2}), one gets
 $$c_{\beta}=\langle\beta+2\,\delta_G,\beta\rangle=\frac{4\,(p+q)}{4\,(p+q+1)}=\frac{p+q}{p+q+1}\,.$$
 \item $\mathrm{F}_4/\mathrm{Spin}(9)$. Here $\beta= e_1$. The half-sum of the positive roots is
 $$\delta_G=\frac{1}{2}\,(11\, e_1+5\, e_2+3\, e_3+e_4)\,.$$
 We considered above the scalar product on weights induced by the usual scalar product on $\bR^4$. In order to compare it with the scalar product induced by the Killing form sign-changed, we use the ``{strange formula}'' of Freudenthal and de Vries (see 47-11 in \cite{FdV}).
 For the scalar product $(\;,\;)$ induced by the usual scalar product on $\bR^4$:
 $$(\delta_G,\delta_G)=39\,,$$
 whereas for the scalar product $\langle\;,\;\rangle$ induced by Killing form sign-changed: 
 $$\langle \delta_G,\delta_G\rangle=\frac{\dim(G)}{24}=\frac{13}{6}\,.$$
 Hence, as the two $\mathrm{Ad}_G$-invariant scalar products have to be proportional since
 $G$ is simple, 
 $$ \langle\;,\;\rangle= \frac{1}{18}\,(\;,\;)\,,$$
 hence 
 $$c_{\beta}=\langle\beta+2\,\delta_G,\beta\rangle=\frac{12}{18}=\frac{2}{3}\,.$$
 
\end{enumerate}

\bibliographystyle{amsalpha}
\bibliography{diracbib}

\providecommand{\bysame}{\leavevmode\hbox to3em{\hrulefill}\thinspace}
\providecommand{\MR}{\relax\ifhmode\unskip\space\fi MR }
\providecommand{\MRhref}[2]{%
  \href{http://www.ams.org/mathscinet-getitem?mr=#1}{#2}
}
\providecommand{\href}[2]{#2}
\begin{thebibliography}{}

\end{thebibliography}


\begin{thebibliography}{BHM+15}
 \bibitem[BDS49]{BdS}
A.~Borel and J~. De~Siebenthal, \emph{Les sous-groupes ferm\'es de rang maximum
  des groupes de {Lie} clos}, Commentarii mathematici Helvetici \textbf{23}
  (1949), 200--221.

\bibitem[Bes78]{Bes78}
A.~Besse, \emph{Manifolds all of {Those} {Geodesics} are {Closed}}, Ergebnisse
  der Mathematik und ihrer Grenzgebiete, vol.~93, Springer-Verlag, 1978.

\bibitem[Bes87]{Bes}
\bysame, \emph{Einstein {Manifolds}}, Springer-Verlag, Berlin, 1987.

\bibitem[BHM+15]{BH3M}
Jean-Pierre Bourguignon, Oussama Hijazi, Jean-Louis Milhorat, Andrei Moroianu,
  and Sergiu Moroianu, \emph{{A spinorial approach to Riemannian and conformal
  geometry}}, Monographs in Mathematics, European Mathematical Society, 2015.

\bibitem[BR90]{BR90}
F.~E. Burnstall and J.~H. Rawnsley, \emph{{T}wistor theory for {R}iemannian
  symmetric spaces}, Lecture Notes in Mathematics, vol. 1424, Springer-Verlag,
  Berlin, Heidelberg, New-York, 1990.

\bibitem[Bro64]{Bro64}
G.~Brown, \emph{A remark on semi-simple {Lie} algebras}, Proc. Amer. Math. Soc
  \textbf{15} (1964), 518.

\bibitem[CG88]{CG}
M.~Cahen and S.~Gutt, \emph{Spin {Structures} on {Compact} {Simply} {Connected}
  {Riemannian} {Symmetric} {Spaces}}, Simon Stevin \textbf{62} (1988),
  209--242.

\bibitem[Cha04]{ECh04}
F.~El Chami, \emph{Spectra of the {Laplace} {Operator} on {Grassmann}
  {Manifolds}}, International Journal of Pure and Applied Mathematics
  \textbf{12} (2004), 395--417.

\bibitem[Cha12]{ECh12}
\bysame, \emph{The {Spectrum} of the {Laplace} {Operator} on the {Manifold}
  ${Sp}(n)/{Sp}(q)\times {Sp}(n-q)$}, Indian Journal of Pure and Applied
  Mathematics \textbf{43} (2012), 71--86.

\bibitem[CW76]{CaWol76}
R.~S. Cahn and J.~A. Wolf, \emph{Zeta {Functions} and {Their} {Asymptotic}
  {Expansions} for {Compact} {Symmetric} {Spaces} of {Rank} {One}}, Commentarii
  Mathematici Helvetici \textbf{51} (1976), 1--21.

\bibitem[FdV69]{FdV}
H.~Freudenthal and H.~de~Vries, \emph{Linear {Lie} {Groups}}, Academic Press,
  New York, 1969.

\bibitem[Hal07]{H07}
M.~B. Halima, \emph{Branching rules for unitary groups and spectra of invariant
  differential operators on complex {Grassmannians}}, Journal of Algebra
  \textbf{318} (2007), no.~2, 520--552.

\bibitem[HC11]{HC11}
Y.~Hong and W.~G. Chen, \emph{Eigenvalues of the manifold ${Sp}(n)/{U}(n)$},
  Acta Mathematica Sinica, English Series \textbf{27} (2011), article 2269.

\bibitem[Hum72]{Hum}
J.~E. Humphreys, \emph{Introduction to {Lie} {Algebras} and {Representation}
  {Theory}}, Springer-Verlag, 1972.

\bibitem[IT78]{IT78}
A.~Ikeda and Y.~Taniguchi, \emph{Spectra and eigenforms of the {Laplacian} on
  ${S}^{n}$ and ${P}^{n}(\mathbb{C})$}, Osaka J. Math. \textbf{15} (1978),
  515--546.

\bibitem[Nag61]{Nag61}
T.~Nagano, \emph{On the {Minimum} {Eigenvalues} of the {Laplacians} in
  {Riemannian} {Manifolds}}, Scientific papers of the College of General
  Education, University of Tokyo \textbf{11} (1961), 177--182.

\bibitem[Str80]{Str80b}
H.~Strese, \emph{Spektren symmetrischer {R{\"a}ume}}, Math. Nachr. \textbf{98}
  (1980), 75--82.

\bibitem[Sug62]{Sug62}
M.~Sugiura, \emph{Representations of {Compact} {Groups} {Realized} by
  {Spherical} {Functions} on {Symmetric} {Spaces}}, Proc. Japan Acad.
  \textbf{38} (1962), 111--113.

\bibitem[Tak94]{Tak94}
M.~Takeuchi, \emph{Modern {Spherical} {Functions}}, Translations of
  Mathematical Monographs, American Mathematical Society, 1994.

\bibitem[TK03a]{TK03a}
Gr. Tsagas and K.~Kalogeridis, \emph{The {Spectrum} of the {Laplace} {Operator}
  for the {Manifold} ${SO}(2p+2q+1)/{SO}(2p)\times {SO}(2q+1)$}, Proceedings of
  The Conference of Applied Differential Geometry - General Relativity and The
  Workshop on Global Analysis, Differential Geometry and Lie Algebras, 2001,
  2003, pp.~188--196.

\bibitem[TK03b]{TK03}
\bysame, \emph{The {Spectrum} of the {Symmetric} {Space}
  ${SP}(\ell)/{SU}(\ell)$}, Balkan Journal of Geometry and its Applications
  (BJGA) \textbf{8} (2003), 109--114.

\bibitem[Tsu81]{Tsu}
C.~Tsukamoto, \emph{Spectra of {Laplace}-{Beltrami} operators on
  ${SO}(n+2)/{SO}(2)\times {SO}(n)$ and ${S}p(n+1)/ {S}p(1)\times {S}p(n)$},
  Osaka J. Math. \textbf{18} (1981), 407--426.

\bibitem[vLCL92]{Lie}
M.~A.~A. van Leeuwen, A.~M. Cohen, and B.~Lisser, \emph{{LiE}, {A} {Package}
  for {Lie} {Group} {Computations}}, Computer Algebra Nederland, Amsterdam,
  1992.

\bibitem[Wol72]{Wol}
J.A. Wolf, \emph{Spaces of constant curvature}, Ams Chelsea Publishing,
  University of Calif., 1972.
\end{thebibliography}

\end{document}